\documentclass{amsart}

\usepackage{amsthm}
\usepackage{amsmath}
\usepackage{amstext}
\usepackage{amsfonts}
\usepackage{latexsym}
\usepackage{amscd}

\theoremstyle{plain}
\newtheorem{theorem}{Theorem}[section]
\newtheorem{lemma}[theorem]{Lemma}
\newtheorem{prop}[theorem]{Proposition}
\newtheorem{corollary}[theorem]{Corollary}

\theoremstyle{definition}
\newtheorem{definition}[theorem]{Definition}

\newtheorem{remark}[theorem]{Remark}

\numberwithin{equation}{theorem}
 \usepackage{amssymb}

\usepackage[cmtip,all]{xy}

\newtheorem{conj}[theorem]{Conjecture}

\numberwithin{equation}{section}

\newcommand{\Z}{{\mathbb Z}}
\newcommand{\HH}{{\mathbb H}}
\newcommand\Hom{\mathrm {Hom}}
\newcommand\Aut{\mathrm{Aut}}
\newcommand{\Arond}{\mathcal{A}}
\newcommand{\Crond}{\mathcal{C}}
\newcommand{\Frond}{\mathcal{F}}
\newcommand{\Grond}{\mathcal{G}}
\newcommand{\Krond}{\mathcal{K}}
\newcommand{\Lrond}{\mathcal{L}}
\newcommand{\Orond}{\mathcal{O}}

\newcommand{\rhomp}{\ensuremath{\mathrm{Rhom}^{\cdot}}}
\newcommand{\rhomf}{\ensuremath{\mathrm{R}\mathcal{H}\mathrm{om}^{\cdot}}}

\newcommand\Af{\mathcal{A}{\rm ff}}
\newcommand\perf{{\rm Perf}}
\newcommand\aut{{\rm \Aut}}

\newcommand\daut{{{\rm Aut}^D}}
\newcommand\paut{{\aut \ltimes \pic }}
\newcommand\pic{{\rm Pic}}
\newcommand\pautz{{\aut^\circ \ltimes \pic^\circ}}
\newcommand\beq{\begin{equation}}
\newcommand\ee{\end{equation}}
\newcommand{\N}{{\mathbb N}}

\newcommand{\C}{{\mathbb C}}
\newcommand{\coh}{{\rm Coh}}

\begin{document}

\title{Some remarks on the group of derived autoequivalences}

\author{ROSAY Fabrice}
\address{ Universit\'e de Mons-Hainaut
Institut de Math\'ematiques,
B\^atiment Le Pentagone,
Avenue du Champ de Mars 6,
B-7000 Mons,
BELGIQUE}
\email{fabrice.rosay@umh.ac.be}
\thanks{}

\begin{abstract}
We prove that the neutral component of the group of derived autoequivalences of
a smooth projective variety is the semi-direct product of the neutral component
of its Picard group and its group of automorphisms. We use this result to prove
several results concerning pairs of derived equivalent varieties.
\end{abstract}

\maketitle

\section{Introduction}

After the seminal work of Mukai, Orlov, Kawamata and many other people after
them, it is now clear that for two algebraic varieties having equivalent derived
categories is a very interesting relation. We know that when it is the case, the
two algebraic varieties share many geometric properties, like their dimension,
the order of their canonical sheaf, their canonical ring, to name a few. It is
clear that the study of this relation is closely related to the study of  the
group of (tringulated) autoequivalences of the derived category of a variety.
Thanks to the work of To\"en and Vaqui\'e \cite{ToVa}, this group is in a
natural way an algebraic group scheme. In this paper we precise the structure of
this group. More accurately we determine its neutral component. We can sum up
everything in the following statement:

\begin{theorem}[see theorem \ref{principal}]
Let $k$ be an algebraicaly closed field and $X$ a smooth and projective
$k$-scheme. Let $D(X)$ be the bounded derived category of coherent sheaves on X.
Then the neutral component of the group of triangulated autoequivalence of
$D(X)$ is $\pautz(X)$.
(Where the notation $G^\circ$ stands for the neutral component of the group $G$)
\end{theorem}

The paper is organised as follows, in section 2 we recall some basic definition
about autoequivalences of derived categories and use this to explain a new
construction of the group $\daut$ of triangulated autoequivalences. Then we
prove the main theorem.\\
In section 3 we give several applications of the theorem \ref{principal}, with a
particular emphaze to  the case of abelian varieties.We close this section with
a reformulation of a conjecture of Kawamata on the cardinality  of classes of
derived equivalent algebraic variety.

\section{The algebraic group of derived autoequivalences}

Throughout this section we fix an algebraically closed field $k$ and a smooth
projective variety $X$ over $k$.
We denote by $D(X)$ the bounded derived category of coherent sheaves on $X$. Our
aim is to explain how to associate to $X$ a locally  algebraic
group scheme $\daut_X$ whose points are naturally identified with exact
autoequivalences of $D(X)$. This was first done by To\"en and Vaqui\'e in
\cite{ToVa} in the context of $\rm DG$-categories. Our construction is different
but leads to the same algebraic group.
Our new result is theorem \ref{principal} wich  describes the neutral component
of this group. Applications are discussed in section 3

\subsection{A quick review of integral transform}

Let $S$ be a $k$-scheme and $X\to S$ and $Y\to S$ be two smooth and projective
$S$-schemes.

\begin{definition}
To any object  $\Frond^\cdot \in D(X\times_S Y)$ we associate a functor, called
an integral transform, by the formula:
\beq \begin{array}{cccl}\phi^{\Frond^\cdot}_{X\to Y}:& D(X) &\to &
D(Y)\\
	 & F&\mapsto & Rp_{Y*}(\Frond^\cdot\otimes^L p_{X}^*(F))\end{array}\ee
\end{definition}

When this functor is an equivalence we call it a Fourier-Mukai transform. The
composite of two integral functors is again an integral functor. More precisely
let  $Z$ be a smooth and projective $S$-scheme and  $\Grond^\cdot \in
D(Y\times_S Z)$. Let $p_{X,Y}$ (resp $p_{Y,Z}$, $p_{X,Z}$) the
projection  of $X\times_S Y\times_S Z$ on $X\times_S Y$ (resp $Y \times_S Z$,
$X\times_S Z$). We define
\beq \Grond^\cdot*\Frond^\cdot=Rp_{X,Z*}(p_{X,Y}^*(\Frond^\cdot)\otimes^L
p_{Y,Z}^*(\Grond^\cdot))\ee

\begin{lemma}\label{composition}
The composite $\phi^{\Grond^\cdot}\circ \phi^{\Frond^\cdot}$ is an integral
transform. Its kernel is given by the following formula:

\beq \phi^{\Grond^\cdot}_{Y\to Z} \circ \phi^{\Frond^\cdot}_{X\to Y}\simeq
\phi^{\Grond^\cdot*\Frond^\cdot}_{X\to Z}\ee
\end{lemma}

This is well known (e.g. see \cite{BoOr}). \\
An integral functor always have a right adjoint.

\begin{lemma}\label{adjoint}
With notation as above the integral functor with kernel $\Frond^\cdot$ has a
right adjoint, the integral transform with kernel
$$\Frond^{\cdot-1}=\rhomf(\Frond^\cdot,\mathcal{O}_{X\times_S Y})
\otimes^L L{p_{X}}^*(\omega_{X/S}[n])$$ where $\omega_{X/S}$ stands for the
relative canonical sheaf of $X$ over $S$.
\end{lemma}

Finally we recall  that  an integral transform commutes with a base change

Let $h:T \to S$ a morphisme of schemes. With evident notations we have the
commutative diagramme:
\beq \xymatrix{ & X_T\times_T  Y_T \ar[dr]^{p_{Y_T}} \ar[dl]^{p_{X_T}}
\ar[dd]^{h_{X\times_S Y}} &\\
X_T \ar[dd]^{h_X}& &Y_T \ar[dd]^{h_Y}\\
&X\times_S Y\ar[rd]^{p_Y} \ar[dl]^{p_X}& \\
X& &Y}\ee

Let $\Frond^\cdot_T=Lh_{X\times_S Y}^*(\Frond^\cdot)$.

\begin{lemma}\label{basechangefm}
The functors $Lh_Y^* \circ \phi_{\Frond^\cdot}$ and $\phi_{\Frond^\cdot_T} \circ
Lh^*_X$ are isomorphic
\end{lemma}

The proof is a straightforward computation using projection and base change
formulae.

\subsection{The algebraic group of derived autoequivalences}

We now construct a presheaf on the site $\Af/k$ of affine $k$-schemes:
\beq \begin{array}{ccl} \Af/k & \to & \rm{Ens}\\
			A& \mapsto & \{\text{perfect complexes }\Frond^\cdot\in
D_{qc}(X\times X\times A)| \Frond^\cdot* \Frond^{\cdot-1}=
\Frond^{\cdot-1}*\Frond^\cdot=\Orond_{\Delta} \}
     \end{array}
\ee
Let's recall that a perfect complex on a scheme $X$ is one locally isomorphic
(in the derived category) to a bounded complex of locally free sheaves of finite
rank and that as usual $\Orond_{\Delta}$ stands for the structure sheaf of the
diagonal embeding $X \hookrightarrow X\times X$.
\begin{theorem}\label{algebraicgroup}
The sheaf $\daut$ associated to the presheaf above (for the etale topology) is
an algebraic space locally of finite presentation. Moreover we can endow it with
the composition law $*$ of \ref{composition}, making it into  a group object in
the category of algebraic spaces over $k$, hence a locally algebraic group
scheme.
\end{theorem}

We don't give a proof of this theorem here, instead we just explain the (simple)
idea of the proof. The reader interested in a proof can consult either my thesis
\cite{these}, or the paper \cite{ToVa} where this result is proven (in a quite
different context). Before we go on we need a last result due to Inaba,
Liebliech and To\"en and Vaqui\'e(see \cite{In}, \cite{Li}, \cite{ToVa}).
Let $\Frond^\cdot$ be a perfect complex in $D_{qc}(X\times S)$. For any point
closed point $s$ of $S$ we denote by $i_s$ the closed immersion $s
\hookrightarrow S$ and $X_s$ the fiber.

\begin{definition}
We say that $\Frond^\cdot$ is simple iff for any closed point $s$ of $S$ the
complex $Li^*_s(\Frond^\cdot)$ is simple, that is:
$$\Hom(Li^*_s(\Frond^\cdot),Li^*_s(\Frond^\cdot))=k$$
We say that $\Frond^\cdot$ is universally gluable (rigid in the terminology of
\cite{ToVa}) iff:
$$\rhomp(Li^*_s(\Frond^\cdot),Li^*_s(\Frond^\cdot)) \text{ is concentrated in
degrees greater or equal to 0}$$
\end{definition}

\begin{theorem}(\cite{In},\cite{Li}, \cite{ToVa})
The etale sheaf associated to the presheaf of perfect, simple, universally
gluable complexes is an algebraic space locally of finite presentation. 
We denote it $\perf_X$
\end{theorem}

\begin{proof}(Idea of the proof of theorem \ref{algebraicgroup})
Let $\Frond^\cdot\in D_{qc}(X\times X)$ a perfect complex such that
$\Frond^\cdot* \Frond^{\cdot-1}= \Frond^{\cdot-1}*\Frond^\cdot=\Orond_{\Delta}$.
Then keeping in mind
that the integral transform associated to $\mathcal{O}_{\Delta}$ is the identity
it follows that $\phi^{\Frond^\cdot}$ is an equivalence. The same argument works
if we base change to $X$ as in the lemma \ref{basechangefm}. Then a direct
computation gives:

$$\phi^{\Frond_X^\cdot}(\mathcal{O}_{\Delta})=\Frond^\cdot$$
So that
$$\Hom^{i}(\Frond^\cdot,\Frond^\cdot)=\Hom^{i}(\mathcal{O}_{\Delta},\mathcal{O}_
{\Delta}).$$

The right member of this last equation is $0$ for $i<0$ and $k$ for $i=0$. So
that we have morphism $j:\daut_X \to \perf_{X\times X}$, induced by an obvious
monomorphism at the level of the presheaves defining these two sheaves.
As sheafification is left exact it follows that the morphism $j$ is a
monomorphism. The idea of the proof is then to show that it is formaly smooth.
It implies that $j$ is an open immersion, thus proving the representability of
$\daut(X)$. To end the proof we appeal to a result of Artin to the effect that a
group object in the category of algebraic spaces over a field is a group scheme
(\cite{ArAn}, lemma 4.2).
\end{proof}

\begin{remark}
\begin{itemize}
\item By a result of Orlov \cite{OrK3}, it is straight forward to check that the
$k$-points of $\daut_X$ are in one to one correspondence with the exact
equivalence of $D(X)$.
\item By a deep result in \cite{ToVa}, this algebraic group has countably many
connected components, we will come back on this latter.
\end{itemize}
\end{remark}

We now come to our main result.

\subsection{The neutral component of $\daut_X$}
Let $\aut$ be the group scheme of  $k$-automorphisms of $X$ and $\pic$ the
Picard scheme of $X$. Let $S \to k$ be an affine scheme. Let's consider
an invertbile sheaf $\Lrond$ on $X_S$ and an $S$-automorphism $f$ of $X_S$. To
this data we asscociate
a sheaf  $\Frond^\cdot=\Psi_S(f,\Lrond)$ on $X\times X \times S$ in the
following maner. Identifying $X\times X \times S$ with the scheme $X_S \times _S
X_S$, we denote by
$\Gamma_f$ the section of $X_S\times _S X_S \overset{p_1}{\to}  X_S$ induced by
$f^{-1}$, we then define
\beq \Frond^\cdot=\Psi_S(f,\Lrond)=\Gamma_{f*}(\Lrond) \ee
It's clear that $\Frond^\cdot$ is perfect  ($X$ is smooth). A straightforward
computation
then shows that the integral transform associated to  $\phi^{\Frond^\cdot}$ is
nothing but  $Rf^{-1}_*(\Lrond\otimes^L())$ which is an equivalence.\\
If  $(f_1,\Lrond_1)$ and $(f_2,\Lrond_2)$  are as above, the integral transform
associated with
 $\Psi_S(f_1,\Lrond_1)*\Psi_S(f_2,\Lrond_2)$ is
$Rf_{2*}^{-1}(Rf_{1*}^{-1}(\Lrond_1)\otimes^L\Lrond_2)$; by the projection
formula we thus have
\beq \Psi_S(f_1,\Lrond_1)*\Psi_S(f_2,\Lrond_2)=\Psi_S(f_1\circ
f_2,\Lrond_1\otimes f^{-1*}_1(\Lrond_2))\ee

So that $\Psi_S$ induce a group morphisme between the semi-direct product of the
 picard group by the group of  $S$-automorphisms (the former acting on the first
by  $f \mapsto f^{-1*}$) on the one side, and the group of invertible integral
transform on the other side. As $\Psi_S$ is injective we can sum up everything
in the

\begin{prop}
The group morphisms $\Psi_S$ induce a  sheaf monomorphism
\beq \Psi:\aut \ltimes \pic \to \daut \ee
which is a group morphism.
\end{prop}

\begin{prop}
The morphism $\Psi$ is open.
\end{prop}

\begin{proof}
As $\Psi$ is a monomorphism between schemes of finite type over a field, it's
enough to check that $\Psi$ is formally smooth. 

So that we have to prove that if
\begin{itemize}
\item $A$ is a local artinian ring with maximal ideal $\mathfrak{m}$
\item $A' \to A$ is an infinitesimal extension, i.e the kernel $I$ is a square
zero principal ideal
\item $(f,\Lrond)\in \paut(A)$ and $\Frond^\cdot \in \daut(A')$ are such that
$\Frond^\cdot\otimes^L_{A'} A=\Psi(f,\Lrond)$
\end{itemize}
then there exist $(f',\Lrond')\in \paut(A')$ such that
$\Frond^\cdot=\Psi(f',\Lrond')$.\\
Its clear that  $\Psi(f,\Lrond)$ is a sheaf flat over  $A$. Let $Z_A$ s
be its support. Lemma \ref{torfini} implies that  $\Frond^\cdot$ is in fact
a sheaf flat over $A'$. Let  $Z_{A'}$  be the support of $\Frond^\cdot$. The
projection $p_1:X_{A'}
\times _{A'} X_{A'} \to X_{A'}$ induce a  morphism $Z_{A'} \to X_{A'}$ whose
reduction
to $A$ is $p_1:Z_A \to X_A$. Hence $p_1:Z_{A'}
\to X_{A'}$ is also an isomorphism.  on its support $\Frond$ is a deformation of
$\Lrond$ which is locally free of rank  $1$. As $\Frond$ is flat over $A'$, it
follows that  $\Frond$ restricted to its support is also free of rank $1$.  To
sum up there exist a morphism $f'$ and an invertible sheaf $\Lrond'$ such that
$\Frond^\cdot=\Psi(f',\Lrond')$ with the slight abuse that we don't know yet
that $f'$ is an automorphism. But it's clear as  $f'$ lifts $f$.
\end{proof}

\begin{lemma}\label{torfini}(\cite{Breq},lemma 4.3)
Let $\pi:S\to T$ be a morphism of schemes, and for each
point $t\in T$, let $i_t:S_t\to S$ denote the inclusion of the fibre
$\pi^{-1}(t)$. Let $\Frond^\cdot$
be an object of $D(S)$, such that for all $t\in T$, $L i_t^*(\Frond^\cdot)$ is
a sheaf on $S_t$. Then $\Frond^\cdot$ is a sheaf on $S$, flat over $T$.
\end{lemma}

As a corollary we get the following theorem summing up all the previous  
discussion

\begin{theorem}\label{principal}
There is an exact sequence of locally algebraic group $$0 \to \pautz \to \daut
\to G \to 0.$$
The group $G$ has at most countably many elements.
The neutral component of $\daut$ is $\pautz$.
\end{theorem}

\begin{remark}
The existence of $\daut$ and the fact that there are countably connected
components in it are already proven in  \cite{ToVa}.
\end{remark}

\begin{proof}
As  $\paut$ is an open subgroup of $\daut$, they have the same neutral
component, and it's clear that the neutral component of $\paut$ is $\pautz$.  At
this point the only non trivial point remaining is the countability assertion
which results from \cite[corollaries 3.31 and 3.32]{ToVa}).
\end{proof}

In the next section we give several applications of the theorem \ref{principal}

\section{applications}

 \subsection{The group $\pautz$ is a derived invariant}

\begin{theorem}\label{theo1} Let $X$ and $Y$ be two smooth projective varieties
over an algebraically closed field $k$. Suppose we have an equivalence of
triangulated categories : $D^b(\coh(X))\overset{\phi}{\simeq} D^b(\coh(Y))$.
Then $$\aut^\circ_X\ltimes \pic^\circ_X \simeq \aut^\circ_Y \ltimes
\pic^\circ_Y$$
\end{theorem}

\begin{proof}
From \cite[theorem 2.2]{OrK3} it follows that the equivalence is  given by a
Fourier-Mukai transform. Let $\Frond^\cdot$ be its kernel and $\Grond^\cdot$ the
kernel of the inverse transform. One then defines a morphism $\daut_X \to
\daut_Y$  by sending $\Crond^\cdot \in \daut_X(S)$ to 
$\Grond^\cdot*\Crond^\cdot*
\Frond^\cdot$. It's clear that this morphism is an isomorphism. So that $\daut_X
\simeq \daut _Y$, in particular their neutral component are isomorphic. Now use
the fact from theorem \ref{principal} that the neutral components respectively
are $\aut^\circ_X\ltimes \pic^\circ_X$ and $\aut^\circ_Y\ltimes \pic^\circ_Y$.
\end{proof}

This fact was already announced by Rouquier in \cite{Ro}, where the author does
not give a proof.
In the same paper the following theorem is conjectured.

\begin{theorem}\label{theo2} Let $A$ be an abelian variety over $k$
an algebraically closed field of characteristic $0$. Suppose we have an
equivalence of triangulated categories
$D^b(\coh(X))\overset{\phi}\simeq D^b(\coh(A))$, Then $X$ is also an abelian
variety.
\end{theorem}

This theorem has been already proved in \cite{HuNi}, we give a new proof using
theorem \ref{theo1}. Actually our proof shows directly that $X$ is an abelian
subvariety of  $A\times
\hat{A}$. It is interesting to compare with \cite{OrAb}.
\begin{proof}
From theorem \ref{theo1}, $\aut^\circ_X\ltimes
\pic^\circ_X \simeq \aut^\circ_A \ltimes \pic^\circ_A$. It follows that 
$\aut^\circ_X$ and
$\pic^\circ_X$ are abelian varieties. The order of the canonical sheaf
and the dimension are derived invariants (see e.g. \cite[lemme 2.1]{BrMa}). So
$\omega_X$ is trivial and if  $n={\rm
dim}X={\rm dim} A$,  then ${\rm dim}(\aut^\circ_X\ltimes
\pic^\circ)=2n$. It is known that the tangent space of  $\aut^\circ_X$ is
$H^0(X,T_X)$ and the tangent space of  $\pic^\circ_X$ is $H^1(X,\Orond_X)$. As
$\omega_X$ is trivial an easy computation using the symetry of Hodge numbers
(Here we use the hypothesis that the characteristic of $k$ is $0$ ) and Serre
duality shows that  ${\rm dim}\ H^1(X,\Orond_X)={\rm dim}\ H^0(X,T_X)$. As
$\aut^\circ_X$ and $\pic^\circ_X$ are smooth, it follows that  ${\rm dim}
\aut^\circ_X={\rm dim}\pic^\circ_X=n$.
Let  $x$ be a point of  $X$. From lemma \ref{action}, the stabilizer of $x$
under the action $\aut^\circ_X$ is finite. Hence the orbit of $x$ is the whole
$X$. To conclude we use the well known fact that the quotient of an abelian
variety  $A$ by a finite subgroup  of $A$ is again an abelian variety.
\end{proof}

\begin{lemma}\label{action}
Let $X$ be a scheme and $G$ an algebraic subgroup of $\aut(X)$ which acts on
$X$ fixing a point $x$. Then  $G$ is affine. In particular if  $G$ is complete
then  $G$ is finite.
\end{lemma}

\begin{proof}
We have a dual action of  $G$ on  $\Orond_x$ and so on the  $k$-vectorial space
$\Orond_x/\mathfrak{m}_x^{n+1}:=V_n$ ($\Orond_x$ is the local ring at $x$ and
$\mathfrak{m}_x$, is the maximal ideal of this ring.). Let $G_n:=\ker(G \to
GL(V_n))$. We have $\cdots \subset G_n \subset  G_{n-1} \subset \cdots \subset
G_0=G$. More over if $g \in \bigcap_n G_n$ then $g$ is the identity on
$\Orond_x$ because  $\Orond_x$ is separated, $\bigcap_n\mathfrak{m}_{n+1}=0$.
Hence
$\bigcap G_n=1$. Now $G$ being algebraic this sequence is stationnary, hence
there exist $n_0$ such that  $G_{n_0}=1$ and  $G$ is a closed subgroup of
$GL(\Orond_x/\mathfrak{m}_x^{n_0+1})$.
\end{proof}

\subsection{Derived equivalence classes of algebraic varieties}

Let's say that two smooth projective varieties $X$ and $Y$ over a field $k$ are
derived equivalent if there exist an equivalence of triangulated categories
$D^b(X)\simeq D^b(Y)$. In \cite{Ka} Kawamata conjecture the following:

\begin{conj}[conjecture 1.5, \cite{Ka}]\label{finitness}
There are up to isomorphism finitely many smooth projective varieties derived
equivalent to a given one.
\end{conj}

It has been proven in \cite{AnTo} that there are at most countably many derived
equivalence classes when $k=\C$.
We give here a new proof of a slighly better result, namely the only assumption
on the field is that it is algebraically closed.

\begin{theorem}\label{countably}
Let $X$ be a smooth projective variety over a field $k$ algebraically closed.
Then there are at most countably many (up to isomorphism) smooth projective
varieties derived equivalent to $X$
\end{theorem}

Before we proceed to the proof we need a lemma due to Favero in \cite{Fa}.
Suppose we have $X$ and $Y$ derived equivalent. Fix an equivalence $F:D^b(X)\to
D^b(Y)$.
Then has we have seen in the proof of \ref{theo1}, $F$ induces an isomorphism
$F^*:\daut_Y \to \daut_X, G \mapsto F^{-1}\circ G \circ F$.

\begin{lemma}[lemma 4.1 \cite{Fa}]\label{ampleiso}
Let $X$ and $Y$ be a smooth projective varieties.  Let $\mathcal A$ be an ample
line bundle on $Y$, $\tau \in \emph{Aut}(X)$, and $\mathcal L \in \emph{Pic }X$
and suppose we have an equivalence $F: D^b(X) \simeq D^b(Y)$ and $F^*((\bullet
\otimes \mathcal A)) = (\tau, \mathcal L)[r]$ for some $r \in \Z$.  Then $F
\cong (\gamma, \mathcal N) [s]$ for some line bundle $\mathcal N \in
\emph{Pic}(Y)$ , an isomorphism $\gamma: X \tilde{\longrightarrow} Y$, and $s
\in \Z$.
\end{lemma}

Combining the above lemma with theorem \ref{principal} we get:

\begin{lemma}\label{componentiso}
Let $X$, $Y$ and $Z$ be smooth projective varieties. Let $\Arond_Y$ and
$\Arond_Z$ be ample line bundle on $Y$ and $Z$ respectively. Let $C_{\Arond_Y}$
(resp  $C_{\Arond_Z}$) be the connected  component of $\daut_Y$ (resp
$(\daut_Z$) that contains $(()\bullet \otimes \Arond_Y)$ (resp  $(()\bullet
\otimes \Arond_Y)$). Let $G:D^b(X) \simeq D^b(Y)$ and $H:D^b(X) \simeq D^b(Z)$
be derived equivalence. Finally suppose $$G^*(C_{\Arond_Y})=H^*(C_{\Arond_Z}).$$
Then $Y$ and $Z$ are isomorphic.
\end{lemma}

\begin{proof}
The hypothesis imply that $F=H\circ G^-1:D^b(Y) \to D^b(Z)$ is an equivalence
and that $F^*(C_{\Arond_Z})=C_{\Arond_Y}$. Now theorem \ref{principal} implies
that any element in $C_{\Arond_Y}$ is a $(\sigma,\Lrond)$, for some automorphism
$\sigma$ of $Y$ and some line bundle $\Lrond$ on $Y$. Applying lemma
\ref{ampleiso} the conclusion follows.
\end{proof}

\begin{corollary}
The number of smooth projective varieties derived equivalent to $X$ is bounded
by the number of connected components of $\daut_X$.
\end{corollary}

\begin{corollary}
Theorem \ref{countably} holds true.
\end{corollary}

\begin{proof}
By theorem \ref{principal}, there are countably many connected components in
$\daut_X$.
\end{proof}

\subsection{A new conjecture}

Using the previous discussion and a new ingredient we introduce a new
conjecture, we prove equivalent to conjecture \ref{finitness}. We begin with
some notations.
Let $(X_i)_{i\in \N}$ be the set of varieties derived equivalent to $X$. Let
$\Arond_i$ be an ample line bundle on $X_i$ and $F_i:D^b(X) \to D^b(X_i)$ be a
triangulated equivalence. Let $\Krond_i^\cdot\in D^b(X\times X)$ be the kernel
of the Fourier-Mukai transform $F_i^*(()\otimes \Arond_i)$. Let $\Lrond$ be a
very ample line bundle on $X\times X$. Finally let $d\in \N$ such that
$\Lrond\oplus \Lrond^{-1} \oplus \ldots \oplus \Lrond^{-d}$ is a strong 
generator of $D^b_{qc}(X\times X)$, recall from \cite{BoVa} that such a $d$
exist.  Then conjecture \ref{finitness} is equivalent to the following

\begin{conj}\label{finitness2}
We can choose the $\Arond_i$ and the $F_i$ such that there exist a function
$\nu:\Z \to \N$ with finite support such that \beq{\rm dim
}\,\HH^j(X,\Krond_i^\cdot \otimes \Lrond^l)\leq \nu(j)\ \forall j, \forall l\in
[0,d], \forall i \label{qcompact}\ee
\end{conj}

\begin{proof}
It's clear that assuming \ref{finitness}, conjecture \ref{finitness2} holds.
Only the converse needs a proof.
So now assume \ref{finitness2} holds. Then from lemma \ref{componentiso} it's
enough to prove that there are only finitely many elements in
$\{C_{\Arond_i}\}$, the set of connected components of $\daut_X$ containing the 
$F_i^*(()\otimes \Arond_i)$. Now the condition (\ref{qcompact}) in conjonction
with proposition \ref{qcompact2} implies 

the family $\Krond_i$ is bounded hence quasi-compact.
\end{proof}

\begin{prop}[corollary 3.32 in \cite{ToVa}]\label{qcompact2}
Let $k$ be a field and $X$ be a smooth and projective
variety over $k$. Let $\mathcal{O}(1)$ be a very ample
line bundle on $X$. Then, there exists an integer $d$, such
that the following condition is satisfied:

A family of perfect complexes $\{E_{i}\}_{i\in I}$ on $X$ is bounded
if and only if there exists a function $\nu : \mathbb{Z} \longrightarrow
\mathbb{N}$
with finite support, such that
$$Dim_{k}\mathbb{H}^{k}(X,E_{i}(j)) \leq \nu(k) \qquad \forall k\; \forall
j\in [0,d] \; \forall i\in I.$$
\end{prop}

\begin{remark}
\begin{itemize}
 \item Any integer $d$ such that $\bigoplus_{i=0}^d \mathcal{O}(-i)$ is a
compact generator is good for proposition \ref{qcompact2}.
\item It follows from \cite{Ro2} that one can choose $d\leq 2{\rm dim}\, X$.
\end{itemize}
\end{remark}


\begin{thebibliography}{99}


\bibitem{ArAn} M. Artin, {\it Algebraization of formal moduli, I.},
 Global Analysis (Papers in Honor of K. Kodaira), (Univ. of Tokyo Press, Tokyo),
1971, 21–71

\bibitem{AnTo} M. Anel and B. To\"en,{\it On derived equivalence classes of
algebraic varieties.}, {\tt 	arXiv:math/0611545v3 [math.AG]}

\bibitem{Breq} T. Bridgeland,{\it Equivalences of Triangulated Categories and
Fourier-Mukai Transforms.},Bull. London Math. Soc. {\bf 31}, (1999) , 25-34.

\bibitem{BrMa}
T.~Bridgeland and A.~Maciocia, {\it   Complex surfaces with equivalent derived
  categories}, Math. Z. \textbf{236} (2001), no.~4, 677-697.

\bibitem{BoOr}
A.~Bondal and D.~O. Orlov, {\it  Semi-orthogonal decompositions for algebraic
  varieties.}, {\tt arXiv:math.AG/9506012}.

\bibitem{BoVa} A. Bondal and M. Van den Bergh, {\it Generators and
representability of functors in commutative and noncommutative geometry}, {\tt 
arXiv:math/0204218v2 [math.AG]}

\bibitem{Fa} D. Favero,{\it Some finiteness results for Fourier-Mukai
partners.}, {\tt 	arXiv:0712.0201v2 [math.AG]}

\bibitem{HuNi} D. Huybrechts, M. Nieper-Wisskirchen,{\it Remarks on derived
equivalences of Ricci-flat manifolds}, {\tt 	arXiv:0801.4747v1 [math.AG]}

\bibitem{In}
Michi-aki Inaba,  {\it Toward a definition of moduli of complexes of coherent
sheaves on a projective scheme.},
 J. Math. Kyoto Univ, 42(2):317-329, 2002.

\bibitem {Ka} Y. Kawamata,{\it D-equivalence and K-equivalence.} ,{\tt
arXiv:math/0205287v3 [math.AG]}

\bibitem{Li} M. Lieblich, {\it Moduli of complexes on a proper morphism.},
J. of Algebraic Geometry {\bf 15}, (2006), 175-206.

\bibitem{OrAb}
 D.~O. Orlov, {\it Derived categories of coherent sheaves on abelian varieties
and equivalences between them.}, Izv. Ross. Akad. Nauk Ser. Mat. \textbf{66}
  (2002), no.~3, 131-158.

\bibitem{OrK3}
D.~O. Orlov, {\it  Equivalences of derived categories and {$K3$} surfaces.}, J.
  Math. Sci. (New York) \textbf{84} (1997), no.~5, 1361-1381, Algebraic
  geometry, 7.

\bibitem{these} F. Rosay, {Sur quelques points de la th\'eorie des
d\'eformations d\'eriv\'ees}, thesis , {\tt
http://www-fourier.ujf-grenoble.fr/\~{}frosay/these.pdf}

\bibitem{Ro} R. Rouquier {\it Cat\'egories d\'eriv\'ees et g\'eom\'etrie
birationelle.},
S\'eminaire Bourbaki, march 2000 to appear in  Ast\'erisque.

\bibitem{Ro2} R. Rouquier {\it Dimensions of triangulated categories.},{\tt
arXiv:math/0310134v3 [math.CT]}



\bibitem{ToVa} Bertrand To{\"e}n and Michel Vaqui{\'e}, {\it Moduli of objects
in
dg-categories.}, {\tt arXiv:math.AG/0503269}

\end{thebibliography}
 \end{document}